\documentclass[final,1p,times]{elsarticle}

\usepackage{amssymb}
\newcommand{\R}{{\Bbb R}}

\newcommand{\C}{{\Bbb C}}

\usepackage{amsmath}

\newtheorem{thm}{Theorem}
\newtheorem{lemma}[thm]{Lemma}
\newtheorem{corollary}[thm]{Corollary}
\newtheorem{proposition}[thm]{Proposition}

\newtheorem{definition}[thm]{Definition}
\newtheorem{remark}[thm]{Remark}
\newproof{proof}{Proof}

\journal{a journal}

\begin{document}

\begin{frontmatter}



\title{A note on the existence of non-monotone non-oscillating  wavefronts}

\author[a]{Anatoli Ivanov}
\author[b]{Carlos Gomez}
\address[a]{
Department of Mathematics,  Pennsylvania State University,
P.O. Box PSU, Lehman, PA 18627, USA
\\ {\rm E-mail: afi1@psu.edu}}

\author[b]{and Sergei Trofimchuk\footnote{Corresponding author.}}
\address[b]{Instituto de Matem\'atica y Fisica, Universidad de Talca, Casilla 747,
Talca, Chile \\ {\rm E-mail: trofimch@inst-mat.utalca.cl}}
\bigskip

\begin{abstract}
\noindent
In this note, we present a monostable delayed reaction-diffusion equation with the unimodal birth function which admits only non-monotone  wavefronts.  Moreover, these fronts are either eventually monotone (in particular, such is  the minimal wave)  or slowly oscillating. Hence, for  the Mackey-Glass type diffusive equations, we answer affirmatively the question  about the existence of non-monotone  non-oscillating  wavefronts.  As it was recently established by Hasik {\it et al.} and Ducrot {\it et al.}, the same question has a negative answer for the KPP-Fisher equation with a single delay.

\end{abstract}

\end{frontmatter}

\newpage

\section{Introduction and main results} \label{intro}

This note deals with the traveling waves for  the diffusive Mackey-Glass type  equation
\begin{equation}\label{17} \hspace{-7mm}
u_t(t,x) = \Delta u(t,x)  - u(t,x) + g(u(t-h,x)), \ \ u(t,x) \geq
0,\ x \in \R^m.
\end{equation}
Population model (\ref{17}) was extensively studied (including its non-local version) during the past decade, e.g. see \cite{FGT, gouss, MLLS, MeiI, TTT} and references therein.  Notice that the non-negativity condition $u(t,x) \geq
0$ of (\ref{17}) is due to  the biological interpretation of $u$ as the size of
an adult population. In this paper we are
mostly concerned with classical positive solutions to (\ref{17}) of the special form $u(t,x) = \phi(ct+ \nu\cdot x),$ $ c >0,$ $ |\nu|=1$, where $\phi$ additionally satisfies the boundary conditions $\phi(-\infty) =0,\ \phi(+\infty)=\kappa$. Such solutions of  equation  (\ref{17}) are called traveling  fronts or simply wavefronts. The function $\phi:\R \to \R_+$ said to be the profile of the wavefront $u(t,x) = \phi(ct+ \nu\cdot x)$.  It is easy to see that each  profile $\phi$ is a positive heteroclinic solution of the delay differential equation
\begin{equation}\label{twe0}
x''(t) - cx'(t)-x(t)+ g(x(t-ch))=0, \quad t \in \R.
\end{equation}
The nonlinear term $g$ in  (\ref{17}) and  (\ref{twe0}) plays the role of  a {\it  birth
function} and therefore it is  non-negative. Motivated by various concrete applications, throughout the paper we
assume that $g$ satisfies the following unimodality condition
\begin{description}
\item[{\rm \bf(UM)}] $g: \R_+ \to  \R_+$ is continuous and has
only one positive local extremum point $x=\theta$ (global maximum).
Furthermore, $g$ has two equilibria $g(0)=0, \ g(\kappa)=\kappa$ with
$g'(0)>1, \ g'(\kappa)<1$ and additionally satisfies $g(x)>x$ for $x\in(0,\kappa)$ and $g(x)<x$ for $x>\kappa$.
\end{description}
Therefore, in view of the terminology used in the traveling waves theory, the diffusive Mackey-Glass type  equation  (\ref{17}) is of monostable type \cite{gouss}.
In the particular case when $g$ is monotone on the interval $[0,\kappa]$ there is quite satisfactory description of all  wavefront solutions for equation (\ref{17}) given by the following result. \begin{proposition} \label{PA1} \cite{LZii,TPT} Suppose that $g$ satisfies {\rm \bf(UM)} and is strictly monotone on $[0,\kappa]$. Then there is
$c_*>0$ (called the minimal speed of propagation) such that equation (\ref{17}) has a unique (up to a translation)  wavefront $u(t,x) = \phi(ct+ \nu\cdot x)$ for each $c \geq c_*$ and every $h \geq 0$. In addition, the profile $\phi$ is a strictly increasing function. If $c <c_*$ then equation (\ref{17}) does not have any wavefront.
\end{proposition}
It is worth noting that  the stability of monotone fronts of (\ref{17}) was successfully analysed in \cite{MLLS,MeiI}.

Now, if $\theta \in (0, \kappa)$ (so that $g$ is not anymore monotone on $[0,\kappa]$), much less information on  the traveling fronts to  (\ref{17}) is available. In particular, as far as we know, for a general function $g$ satisfying the hypothesis {\rm \bf(UM)}, none of the three aspects  (the existence of the minimal speed $c_*$, the uniqueness, the monotonicity properties, the wavefront stability) mentioned in Proposition \ref{PA1} has received a satisfactory characterization.  In this paper, we shed some new light on the description  of  possible  geometric shapes of the wavefront profiles $\phi$. Due to the biological interpretation of solutions to (\ref{17}),  the geometric properties  of leading (invading) parts of wavefront profiles  characterize the `smoothness' of the expansion (invasion) processes.  
This fact  shows the practical importance of our studies. 
 A first  picture of the wavefront monotonicity properties  was obtained in \cite{TTT} under the following additional condition

\begin{description}
\item[{\rm \bf(FC)}]  The restriction $g:
[g^2(\theta), g(\theta)] \to \R_+$ has the positive feedback with respect
to the equilibrium $\kappa$:\, $(g(x) - \kappa)(x - \kappa) < 0,$  $x \not= \kappa$. Here we use the notation 
$g^2(\theta)$ for $g(g(\theta))$. 
\end{description}
More precisely, the following result holds:
\begin{proposition}\cite{TTT}
\label{PA2} Consider the case when {\rm \bf(UM)} holds and $g'(\kappa) < 0$. Let  $u(x,t) = \phi(\nu \cdot x +ct)$ be a  wavefront
to Eq. (\ref{17}).  Then there exists $\tau_1 \in \R \cup \{+\infty\}$  such that $\phi'(s) > 0$ on
$(-\infty, \tau_1)$. Furthermore, $\tau_1$ is finite if and only if
$\phi(\tau_1) > \kappa$.  If, in addition, the birth function $g$  satisfies  {\rm \bf(FC)},  then $\phi$ is eventually either monotone or
slowly oscillating around $\kappa$.  Finally, if $\tau_0$ is the leftmost point where $\phi(\tau_0) = \theta$ then $\tau_1-\tau_0 \geq ch$.  \end{proposition}
It should be noted that the existence of oscillating traveling fronts in  the delayed reaction-diffusion equations  is by now a well-known fact confirmed both numerically  and analytically.  The subclass of slowly
oscillating profiles is defined below:
\begin{definition} Set $\mathbb{K} = [-ch,0] \cup \{1\}$. For any
$v \in C(\mathbb{K})\setminus\{0\}$ we define the number of sign
changes by $$\hspace{-1mm} {\rm sc}(v) = \sup\{k \geq 1:{\rm \it
there \ are \ } t_0 <
 \dots < t_k  \ {\rm \it such \ that\ }
v(t_{i-1})v(t_{i}) <0 {\rm \ for \ }  i\geq 1\}. $$ We set ${\rm
sc}(v) =0$ if $v(s) \geq 0$ or  $v(s) \leq 0$ for $s \in
\mathbb{K}$. If $\varphi(t), t \geq a-ch,$ is a solution of
Eq. (\ref{twe0}), we set \ $(\bar \varphi_t)(s) = \varphi(t+s)-
\kappa$ if $s \in [-ch,0]$, and $(\bar \varphi_t)(1) = \varphi'(t)$.
We say that $\varphi(t)$ is slowly oscillating about $\kappa$
if $\varphi(t)-\kappa$ is oscillatory and for each $t \geq a$, we
have either sc$(\bar \varphi_t)=1$ or sc$(\bar \varphi_t)=2$.
\end{definition}

The studies carried over in \cite{TTT} have left unanswered the conjecture about  the existence of non-monotone
but eventually monotone traveling fronts  for equation (\ref{17}) (in particular, for  the well-known diffusive Nicholson's blowflies equation with $g(x) = px\exp(-x)$).   The new facts that have appeared after publication of  \cite{TTT}  did not give an unconditional support to this conjecture. From one side,  numerical simulations of wavefronts for more general non-local equations (e.g. the non-local KPP-Fisher equation \cite{BNPR}) indicate, in certain cases,  the presence of non-monotone but eventually monotone traveling fronts.  See also \cite{ZAMP, IDD, HK, NPT,ST}. On the other hand, the recent works \cite{DN,HK} establish analytically that the KPP-Fisher equation with a finite discrete delay can have wavefronts only with profiles which are either monotone or slowly oscillating around $\kappa$. It is noteworthy  that the above mentioned results of  \cite{DN,HK} were  predicted in \cite{NRRP}. 

In any event, in the present work we give a rigorous analytical justification of the existence of  the proper eventually monotone wavefronts to equation (\ref{17}), see Fig.\ref{F1} below.   In consequence, we answer affirmatively the conjecture stated in \cite{TTT}.
\begin{figure}[h] \label{F1}
\centering\fbox{\includegraphics[width=13cm]{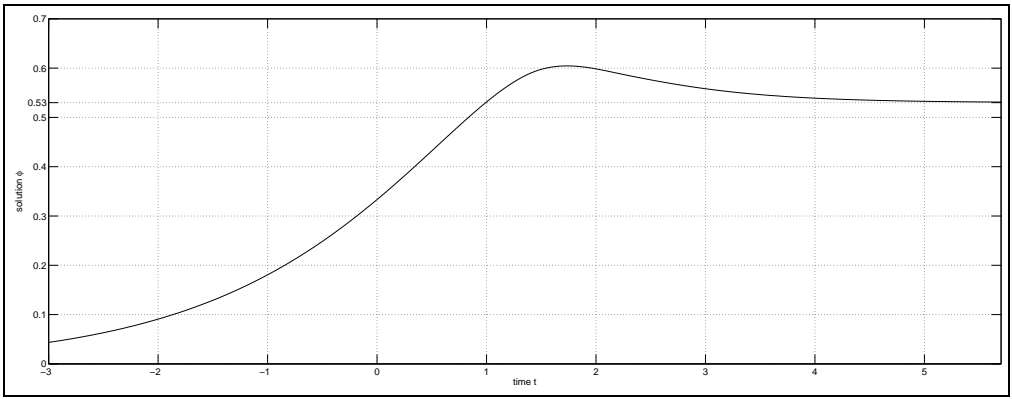}} 
\caption{\hspace{0.1 cm} Profile of a minimal, non-monotone and non-oscillating wavefront to equation (\ref{17}).}
\end{figure}

\noindent Actually, our main result contains even more information:
\begin{thm}\label{main1}
There is a piece-wise linear unimodal function $g$ (see Fig. 2) satisfying   ${\rm \bf(UM)}$, ${\rm \bf(FC)}$  and positive numbers $h, c_*< c^*$  such that the following holds:

a) \ for each $c \geq c_*$ equation (\ref{17}) has a unique wavefront $u(t,x) = \phi(x\cdot \nu +ct), \ |\nu|=1$,  and it does not have any wavefront propagating  with the speed  $c < c_*$;

b) \ for each $c \in [c_*,c^*]$, the profile  $\phi$ is non-monotone but eventually monotone  (see Fig. \ref{F1}, where the minimal front is depicted);

c) \ for each $c > c^*$, the wavefront  profile $\phi$ slowly oscillates around $\kappa$.
\end{thm}
The proof of this theorem combines several ideas from \cite{GT, FGT, TTT}.  It is given in the next section.

\section{Proof of Theorem \ref{main1}}\label{Se1}
A direct  analysis of (\ref{twe0}) shows that each local maximum $M_j = \phi(t_j)$ of the front profile $\phi(t)$
should satisfy the inequality 
$$M_j =  \phi''(t_j)- c\phi'(t_j) + g(\phi(t_j-ch)) \leq g(\theta).$$  Therefore it suffices to consider $g$ defined on  the interval
$[0, g(\theta)]$ only. In the simplest `unimodal' case, the graph of $g$ consists of two linear segments.
This nonlinearity was already analyzed in \cite{TTT}. Since,  in such a case, $g$ satisfies the following sub-tangency condition at $\kappa$:
\begin{equation} \label{sc}
g(x) \leq \kappa + g'(\kappa)(x-\kappa), \quad x \in [0,\kappa],
\end{equation}
 each eventually monotone wavefront is in fact a monotone front, see   \cite{FGT} for more detail.   Therefore, if we want to construct a piece-wise linear birth function $g$ suitable for
 Theorem \ref{main1}, its graph must contain at least three linear  segments and do not satisfy the inequality (\ref{sc}), see  Fig. \ref{F2}:
 \begin{figure}[h]\label{F2}
\centering \fbox{\includegraphics[width=5.5cm]{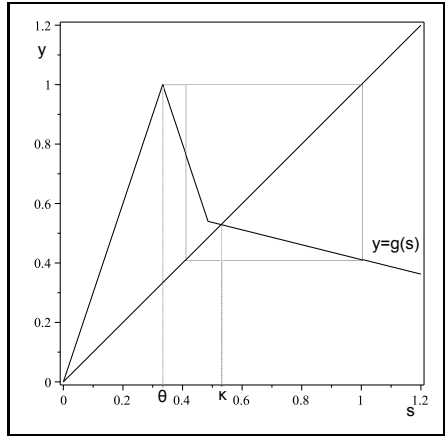}}
\caption{\hspace{0.1cm} Graph of the unimodal birth function $g$  from  Theorem \ref{main1} }
\end{figure}

\begin{equation}\label{dg}
g(x):=\left\{\begin{array}{ll}k_1x,&0\leq x\leq\theta, \\ k_2 x+q_2,&\theta \leq x\leq \theta_1, \\ k_3 x+q_3,& \theta_1 < x \leq g(\theta).  \end{array}\right.
\end{equation}
Here real numbers $q_j$ are chosen to assure the continuity of $g$.
Hence, in what follows, we will seek for the appropriate parameters $k_j, \theta_j$ and $(h,c)$ to obtain the desired shape  of the profile.
Actually, one of the main restrictions on $(h,c)$
was already found in \cite{FGT},  where it was proved that an eventually monotone wavefront in the Mackey-Glass type equation can appear only for $(h,c)$ belonging to the connected closed domain $\mathcal{D}_{\frak L}$ defined below:
\begin{definition}\label{DEF5}
$(h,c) \in \mathcal{D}_{\frak L}$ if and only if each of the  equations $\chi_0(z):= z^2-cz -1 + g'(0)e^{-z ch }=0$,   $\chi_\kappa(z):= z^2-cz -1 + g'(\kappa)e^{-z ch }=0$,  has exactly two real roots (counting the multiplicity) of the same sign: the positive roots $0< \mu_2\leq \mu_1$ for the first equation, and the negative roots $\lambda_2 \leq \lambda_1 <0$ for the second one.
\end{definition}

The following result (established in  \cite[Lemma 1.1]{FGT} and \cite[Lemma 21]{TTT})  partially describes
the structure of the set $\mathcal{D}_{\frak L}$ and other properties of eigenvalues $\lambda_j$:
\begin{lemma} \label{lc2}   Suppose that $g'(\kappa) <0$. Then there exists
$c^*=c^*(h) \in (0, +\infty]$ such
that the characteristic function $\chi_\kappa(z)$
has  three real zeros $\lambda_1\leq \lambda_2 <  0 < \lambda_3$
if and only if $c \leq c^*$.  If $c^*$ is finite and $c=c^*$,  then $\chi_\kappa$ has a double zero $\lambda_1= \lambda_2<0$,
while for $c > c^*$ there does not exist any negative
zero to $\chi_\kappa$.  Moreover,  if $\lambda_j  \in \C$ is a
complex zero of $\chi_\kappa$ for $c \in (0, c^*]$  then $\Re \lambda_j < \lambda_2$ and $|\Im \lambda_j| > 2\pi/(ch)$.
\end{lemma}

By \cite[Theorem 4.5]{TT}, for each $(h,c) \in \mathcal{D}_{\frak L}$, equation (\ref{twe0}) has at least one semi-wavefront solution (i.e. positive bounded solution $\phi(t)$ such that $\phi(-\infty)=0$. As it was established in \cite[Lemma 4.3]{TT}, each semi-wavefront for equation (\ref{17}) has the following separation property $\liminf_{t \to +\infty}\phi(t) >0$. The latter inequality  is also sometimes  considered in the definition of a semi-wavefront, cf. \cite{BNPR}). 
To make this semi-wavefront converge to $\kappa$ at $+\infty$, we will impose an additional condition on $g, h, c$ described in the next proposition.
This condition is given in terms of $g$ and a new piece-wise linear unimodal function $\sigma: [g^2(\theta), g(\theta)] \to [g^2(\theta), g(\theta)]$  defined by $\sigma(x) = \zeta^{-1} ((1-\xi)g(x))$, where
$$
\xi = \xi(h,c)= \frac{z_2 - z_1}{z_2 e^{-chz_1}- z_1 e^{-chz_2}} \in [e^{-h},1],  \quad  \zeta(x)  = x - \psi(x),
$$
$\psi: [g^2(\theta), g(\theta)] \to [\theta_1, g(\theta)]$ is the inverse of $g$ restricted to the interval $[\theta_1, g(\theta)]$,
and $z_1(c) < 0 < z_2(c)$ are the  roots of the equation $z^2-cz-1=0$.
\begin{proposition} \label{PA3} Assume {\rm \bf(UM)} and the following global stability condition
\begin{description}
\item[{\rm \bf(GA)}]   $\kappa$ is the globally attracting fixed point for at least one of the following  two one-dimensional maps
$g, \sigma: [g^2(\theta), g(\theta)] \to [g^2(\theta), g(\theta)]$.
\end{description}
Then every semi-wavefront solution of (\ref{twe0}) converges to $\kappa$  at infinity: $\phi(+\infty) = \kappa$.
\end{proposition}
\begin{proof} A demonstration of this result constitutes the main part of the proof of Theorem 5.1 from \cite{TT}.  
 \hfill $\square$
\end{proof}
\begin{remark}
Note that for  the birth function  $g$ defined by (\ref{dg}) the hypotheses of Proposition \ref{PA3} can be easily verified since the continuous graphs of both $\sigma$ and $g$ are piecewise linear.  In order to verify the hypothesis  {\rm \bf(GA)} in the case of  unimodal  
 $C^3$-smooth birth functions, the authors of \cite{TT} have systematically used  the criterion of the negative Schwarz derivative.  
\end{remark}
The above discussion leads to our first auxiliary  result:
\begin{lemma}\label{fl}
Suppose that  the hypotheses {\rm \bf(UM)}, {\rm \bf(FC)} and {\rm \bf(GA)} are satisfied and that $(h,c) \in \mathcal{D}_{\frak L}$. Then  there exists at least one  traveling front  $u(t,x) = \phi(x\cdot \nu +ct), \ |\nu|=1,$  to  equation (\ref{17}) and its profile $\phi$  must be eventually monotone.
\end{lemma}
\begin{proof}  As we have already mentioned, the existence of at least one semi-wavefront $\phi$ for (\ref{17}) is assured by  \cite[Theorem 4.5]{TT}.  Due to Proposition \ref{PA3},  this semi-wavefront is actually a wavefront. Therefore we only have to prove the eventual monotonicity of $\phi$.  Suppose, on the contrary, that $\phi(t)$ is oscillating around $\kappa$. Since the feedback condition {\rm \bf(FC)} is satisfied, Proposition \ref{PA2} shows that these decaying oscillations should be slow. In addition,
we claim that the convergence of $\phi(t)$ to $\kappa$ is not super-exponential. Indeed, by our construction,  the difference $y(t): = \phi(t) - \kappa$ satisfies the linear homogeneous equation
\begin{equation}\label{ey}
y''(t) - cy'(t) - y(t) +k_3 y(t-ch) =0,
\end{equation}
for all sufficiently large positive  $t$.  Therefore, if $y(t)$ is a small (i.e. super-exponentially decaying) solution of  (\ref{ey}), it should be identically zero for all large positive $t$, see Theorem 3.1 in \cite[p. 76]{hale}. In this way, there exists a leftmost $T$ such that $\phi(t) = \kappa$ for all $t \geq T$. But then, by using equation (\ref{twe0}), we easily get a contradiction since $\phi(t) = \kappa$ for all $t \geq T-ch$.

Now, since $y(t)$ is not a small solution of (\ref{ey}), it can be approximated by a finite linear combination of the eigenfunctions
$$
y(t) = a_1e^{\lambda_1 t} + a_2 e^{\lambda_2 t} +a_j e^{\Re \lambda_3 t}\sin (\Im \lambda_j t + a_4) + O(e^{(\Re \lambda_3-\delta)t}),
$$
where $a_1, a_2, a_j \in R$ and $a_j \not =0$, and $\delta >0$ is sufficiently small, e.g. see \cite[Theorem 6.7]{BC}. From our assumption about the oscillatory
behavior of $\phi$, we deduce that actually $a_1=a_2 =0$. Recalling now that $\Im |\lambda_ j |> 2\pi/(ch) $, we obtain that $\phi(t) = \kappa + a_j e^{\Re \lambda_3 t}\sin (\Im \lambda_j t + a_4) + O(e^{(\Re \lambda_3-\delta)t})$ is rapidly oscillating about $\kappa$, a contradiction. \hfill $\square$
\end{proof}
Before announcing our next result, we recall that, by Proposition \ref{PA2},  the leading part of the wavefront is monotone between the equilibria. Since, in addition,  $\phi_s(t):=\phi(t+s)$ solves (\ref{twe0}) for each fixed $s$, there is no loss of generality in assuming that $\phi(0)=\theta \in (0,\kappa)$ and that $\phi'(t) >0$ for all $t \leq 0$. As a consequence, $\phi(t)$ satisfies the  linear homogeneous equation
\begin{equation}\label{ey2}
y''(t) - cy'(t) - y(t) +k_1 y(t-ch) =0,
\end{equation}
for all $t \leq ch$. This fact allows us to find an almost complete representation of $\phi$ for $t \leq ch$ described by the following lemma.
\begin{lemma}\label{fl2} In addition to all assumptions of Lemma \ref{fl}, 
suppose that $\phi(0) = \theta$ and that  $\mu_2 \leq \mu_1$ are as in Definition \ref{DEF5}. Let the unimodal continuous function $g$ defined by (\ref{dg}) have the shape as shown in Fig. 1. Then, for all $t \leq ch$,   it holds that
\begin{equation}\label{rep1}
\phi(t) = pe^{\mu_2 t} + (\theta -p) e^{\mu_1 t}\ \ \mbox{ if}  \ \mu_2 < \mu_1,  \quad \phi(t) = - qt e^{\mu_1 t} + \theta e^{\mu_1 t}\ \mbox{ if}  \ \ \mu_2 = \mu_1,
\end{equation}
for some  $p, q $ satisfying the inequalities
\begin{equation}\label{rep2}
\theta < p \leq \frac{\mu_1\theta}{\mu_1 - \mu_2e^{-ch(\mu_1-\mu_2)}}, \quad 0<q \leq \frac{\mu_1\theta}{1+ \mu_1ch}.
\end{equation}
\end{lemma}
\begin{proof}  Since $\phi(-\infty) =0$ and since $\phi(t)$ is not a small solution at $-\infty$ by Theorem 3.1 in \cite[p. 76]{hale}, we find that  $\phi$ can be represented by a finite sum
$$
\phi(t) = \sum_{\Re \mu_j >0} c_je^{\mu_j t}, \ t \leq 0, \ \mbox{ if}  \ \mu_2 < \mu_1,  \quad  \phi(t) = c_0te^{\mu_1 t} + \sum_{\Re \mu_j >0} c_je^{\mu_j t}, \ t \leq 0, \  \mbox{ if}  \ \ \mu_2 = \mu_1,
$$
where $\mu_j$ are roots of the characteristic equation $z^2 -cz -1 + g'(0)e^{-zch}=0$ with the positive
real parts (it is a well known fact that the set of all such roots is finite, cf. \cite{hale}).  Now, since $\Re\mu_j < \mu_2 \leq \mu_1$ for each $j < 2$, we find that, in fact,
$$\phi(t) = c_2e^{\mu_2 t} + c_1e^{\mu_1 t},\ t \leq 0, \ \mbox{ if}  \ \mu_2 < \mu_1,  \quad  \phi(t) = c_0te^{\mu_1 t} + c_1e^{\mu_1 t}, \ t \leq 0, \  \mbox{ if}  \ \ \mu_2 = \mu_1.
$$ Indeed, otherwise
$\phi(t)$ will oscillate at $-\infty$.  Taking into account  that $\phi(0)=\theta,$ we obtain formulas (\ref{rep1}).

Next, in order to prove the first inequality for $p$ in (\ref{rep2}), we observe that the coefficient $c_1:= \theta- p$
can be calculated explicitly (e.g. see \cite[Lemma 28]{GT}) with the help of the bilateral Laplace transform:
$$
(\theta - p)e^{\mu_1 t} = - {\rm Res}_{z=\mu_1}\left[ \frac{e^{zt}}{\chi_\kappa(z)}\int_{-\infty}^{+\infty}e^{-zs}f(s)ds \right],
$$
with  $f(s):= g'(0)\phi(s-ch)- g(\phi(s-ch)) \geq 0,\ s \in \R,$ satisfying $f(+\infty)= (g'(0)-1)\kappa >0$ and
$\chi_\kappa(z)= z^2-cz-1+k_1e^{-zch}$.  In consequence, since $\mu_1$ is a simple zero of $\chi_\kappa$ and $\chi'(\mu_1) >0$, we find that
$$
\theta - p = -  \frac{1}{\chi_\kappa'(\mu_1)}\int_{-\infty}^{+\infty}e^{-\mu_1s}f(s)ds <0.
$$
Finally, Proposition \ref{PA2}  guarantees that $\phi'(t) >0$ for all $t \in [0, ch)$, see also \cite[Lemma 10]{TTT}.  In particular, $\phi'(ch)\geq 0$  which amounts to the second inequality for $p$  in (\ref{rep2}).

Using the obtained restrictions on $p$, we easily find that, if  $\mu_2 < \mu_1,$ then
\begin{equation}\label{inf}
\inf_{p}\max_{t\in [0,ch]}\phi(t,p) = \inf_{p} \phi(ch,p) = \frac{g(\theta)}{1+\mu_1\mu_2},
\end{equation}
where $\phi(t,p): =  pe^{\mu_2 t} + (\theta -p) e^{\mu_1 t}$ and the
$\inf$ is taken over the admissible interval for $p$ given by (\ref{rep2}).  In particular, each non-minimal  wavefront should satisfy (\ref{inf}).

Similarly, if $\mu_1=\mu_2$, we obtain
$$
(- qt + \theta)e^{\mu_1 t} = - {\rm Res}_{z=\mu_1}\left[ \frac{e^{zt}}{\chi_\kappa(z)}\int_{-\infty}^{+\infty}e^{-zs}f(s)ds \right],
$$
and therefore
$$q=   \frac{2}{\chi_\kappa''(\mu_1)}\int_{ch}^{+\infty}e^{-\mu_1s}f(s)ds >0.$$
Now, the inequality for $q$ in (\ref{rep2}) is equivalent to the above mentioned property $\phi'(ch) \geq 0$ satisfied by each wavefront.
 \hfill $\square$
\end{proof}
\begin{corollary} \label{crot10} Let all assumptions of Lemma \ref{fl2} be satisfied and $c > c_*$. Then
\begin{equation} \label{sic2}
\phi(ch, c)   \geq \frac{g(\theta)}{1+\mu_1(c)\mu_2(c)}.
\end{equation}
\end{corollary}
\begin{proof}  To prove (\ref{sic2}), it suffices to use the expression for $\phi$ in (\ref{rep1}) and the inequality for $p$ in (\ref{rep2}):
$$
 \phi(ch, c) = p(e^{\mu_2 ch}-  e^{\mu_1 ch})  + \theta e^{\mu_1 ch}\geq \frac{\mu_1\theta}{\mu_1 - \mu_2e^{-ch(\mu_1-\mu_2)}}(e^{\mu_2 ch}-  e^{\mu_1 ch}) + \theta e^{\mu_1 ch}=\frac{g(\theta)}{1+\mu_1(c)\mu_2(c)}.
$$
 \hfill $\square$
\end{proof}
\begin{corollary} \label{crot11}Assume, in addition to conditions of Lemma \ref{fl2},  that each admissible wavefront to equation (\ref{17}) is unique (up to translation). Then
\begin{equation} \label{sic}
\phi(c_*h, c_*)   \geq \frac{g(\theta)}{1+\mu_1^2(c_*)}, \quad 0< q \leq \frac{\theta-g(\theta) e^{-\mu_1(c_*) c_*h}/(1+\mu_1^2(c_*))}{c_*h}.
\end{equation}
\end{corollary}
\begin{proof}   Let  $u(t,x) = \phi(x+ct, c), \ \phi(0,c)=\theta,$ be the wavefront propagating at the velocity $c > c_*$. It is easy to see that each profile $\phi(t,c)$ satisfies the integral equation
\begin{equation}\label{ie1}
\phi(t,c) = \frac{1}{z_2
- z_1} \left(\int^t_{-\infty} e^{z_1 (t-s)}g(\phi(s-ch, c))ds +
\int_t^{+\infty}e^{z_2 (t-s)}g(\phi(s-ch,c))ds\right),
\end{equation}
where $z_1< 0<z_2$ are roots  of the  equation $z^2 -c z -1 =0$.  As a consequence, 
$$
\phi'(t,c) = \frac{1}{z_2
- z_1} \left(z_1\int^t_{-\infty} e^{z_1 (t-s)}g(\phi(s-ch, c))ds + z_2
\int_t^{+\infty}e^{z_2 (t-s)}g(\phi(s-ch,c))ds\right),
$$
and a straightforward estimation shows that $|\phi'(t,c)| \leq g(\theta)/\sqrt{c^2+4}$ and $|\phi(t,c)| \leq g(\theta)$. 
Choose now a strictly decreasing sequence $c_j \to c_*$. By the Arzel\`a-Ascoli theorem, the sequence $\phi(t,c_j)$ has a subsequence $\phi(t,c_{j_k})$ which converges, uniformly on compact subsets of $\R$, to the continuous non-negative bounded function $\phi_0(t)$. It is clear that $\phi$ is non-decreasing on $(-\infty, ch]$ and that $\phi_0(0) = \theta$.  By the Lebesgue«s dominated convergence theorem, $\phi_0$ satisfies equation (\ref{ie1}) with $c=c_*$ and therefore  $\phi_0$ is a non-negative profile of a traveling wave propagating with the velocity $c_*$. 
Since the limit value $\phi_0(-\infty) < \theta$ must satisfy equation (\ref{twe0}), we get $\phi_0(-\infty) =0$. This means that actually  $\phi_0(t)$ is a wavefront.  But then,  due to the uniqueness assumption, we have that $\phi_0(t) = \phi(t,c_*)$ and that
$$(-qc_*h+ \theta)e^{\mu_1c_*h} =\phi(c_*h,c_*) = \phi_0(c_*h,c_*) = \lim_{j \to +\infty} \phi(c_jh,c_j) \geq  \lim_{j \to +\infty}  \frac{g(\theta)}{1+\mu_1(c_j)\mu_2(c_j)} = \frac{g(\theta)}{1+\mu_1^2(c_*)}.$$
The inequalities (\ref{sic})  follow easily from  these relations.
 \hfill $\square$
\end{proof}
The above considerations yield the following conclusion:
\begin{thm} \label{main2}
Let the unimodal continuous function $g$ be defined by (\ref{dg}), where $k_2 < k_3 <0 < 1< k_1$.   In addition, suppose that  the hypotheses {\rm \bf(FC)} and {\rm \bf(GA)} are satisfied,  $(h,c) \in \mathcal{D}_{\frak L}$,  while
\begin{equation} \label{11}
\gamma(c):= \frac{g(\theta)}{1+\mu_1(c)\mu_2(c)}> \kappa.
\end{equation}
Then equation (\ref{17})  has a non-empty set of traveling fronts propagating with the speed c (which can be the minimal one). Next,
each such wavefront is non-monotone on $\R$ but eventually monotone.  Furthermore, if either $|k_2| \leq k_1$ or the characteristic equation $z^2-cz-1 +|k_2|e^{-zch} =0$ has two real positive roots (counting multiplicity), then there exists a unique  (up to a translation) wavefront propagating with the velocity $c$.
\end{thm}
\begin{proof} By Lemma \ref{fl}, there exists at least one  traveling front  $u(t,x) = \phi(x\cdot \nu +ct), \ |\nu|=1,$  to  equation (\ref{17}) and its profile $\phi$  must be eventually monotone.  On the other hand, Corollaries \ref{crot10}, \ref{crot11} and inequality (\ref{11}) assure that $\phi(ch) > \kappa$ and therefore the profile $\phi$ is non-monotone. Finally,
since
$|g(s_1)- g(s_2)| \leq \max\{k_1, |k_2|\}|s_1-s_2|, \ s \in [0,1]$,   the uniqueness (up to a translation) of the wavefront propagating with the given velocity $c$ follows from \cite[Theorems 7,8]{AGT}.
 \hfill $\square$
\end{proof}
{\bf Proof of Theorem \ref{main1}}:
 Set $k_1=-k_2=3, \ k_3 = - 0.25, \ \theta =1/3, \  h = 2, \ \kappa = 0.53$. Then the minimal speed $c_* = 0.712\dots $ and the critical speed $c^*=0.751\dots$ can be found from the characteristic equations
$z^2 -cz-1+3 e^{-2cz}=0, \ z^2 -cz-1-0.25 e^{-2cz}=0$. Recall that, by definition, $\{2\}\times [c_*,c^*] =   \mathcal{D}_{\frak L} \cap \{2\}\times \R_+$. We also have that $\mu_1(c_*) = \mu_2(c_*) = 0.926\dots$.   It is a well known fact (cf. \cite[Theorem 1.1]{TT}) that equation (\ref{17}) does not have any semi-wavefront  propagating at the velocity 
$c < c_*$. 
Next,  a straightforward (but a little bit tedious) evaluation of $\gamma(c)$ shows that inequality (\ref{11}) holds for each $c \in [c_*,c^*]$. For the completeness,  below we present the proof of this fact:
\begin{lemma} Consider the above defined $g$ and let $h=2$. Then
$$
\gamma(c) > \gamma_1(c): = \frac{1+1.53c^2}{2.55+1.53c^2} >\kappa = 0.53, \quad c \in [c_*,c^*]= [0.712\dots, 0.751\dots].
$$
\end{lemma}
The graph of the function $\gamma_1: [c_*,c^*]  \to \R_+$ is shown in Fig. 3.
\begin{figure}[h]\label{F3}
\centering \fbox{\includegraphics[width=5.5cm]{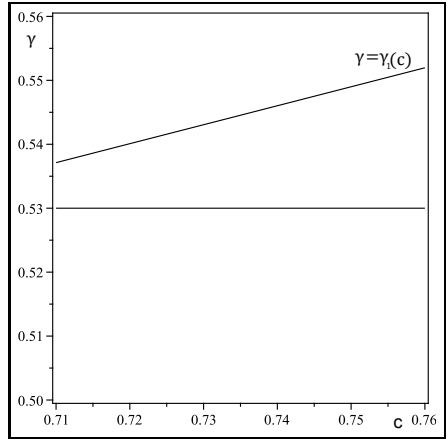}}
\caption{\hspace{0.5cm} Graph of $\gamma_1(c)$ }
\end{figure}
\begin{proof}  Set $\rho_j(c) = c\mu_j(c)$, then $0< \rho_1(c) \leq \rho_2(c)$ are the only two real roots of the equation
$
1+z- c^{-2}z^2  = 3e^{-2z}.
$
A direct computation shows that $\rho_1(c_*) =\rho_2(c_*)= 0.656\dots$ and
$\rho_1(c^*) = 0.537\dots, \ \rho_2(c^*) = 0.867\dots$. Now, for each fixed $z \in \R$ the function $P(c):= 1+z- c^{-2} z^2 $ is strictly increasing on $(0, +\infty)$, and therefore $\rho_1(c)$ is strictly decreasing and $\rho_2(c)$ is strictly increasing on $ [c_*,c^*]$. In particular,  $\rho_1(c), \rho_2(c) \in [\rho_1(c^*), \rho_2(c^*)]\subset [0.537, 0.868]$ for all $c \in [c_*,c^*]$.  Next, let us consider the quadratic polynomial
$$
Q(z)=3\cdot e^{-2 \rho_1(c_*) }\left( 1-2.04(z-\rho_1(c_*))+1.9( z-\rho_1(c_*))^{2} \right),
$$
which is a small deformation of the second order Taylor approximation of the function $y= 3e^{-2z}$ at
$z=\rho_1(c_*)$.  It can be easily verified that $Q(z) > 3e^{-2z}$ for all  $z \in [0.521, \rho_1(c_*))$ and
$Q(z) < 3e^{-2z}$ for all $z \in (\rho_1(c_*), 0.885]$.  As a consequence, for each $c \in [c_*,c^*]$,
the equation
$
1+z- c^{-2}z^2  = Q(z)
$
has exactly two real roots $\tilde z_1(c) > \rho_1(c)$,   $\tilde z_2(c) > \rho_2(c)$.
Therefore $$\mu_1(c)\mu_2(c) = c^{-2}\rho_1(c)\rho_2(c) < c^{-2}\tilde z_1(c)\tilde z_2(c) = \frac{Q(0)-1}{1+5.7c^2e^{-2 \rho_1(c_*)}}= \frac{1.549\dots}{1+5.7c^2e^{-2 \rho_1(c_*)}}$$ and
$\gamma(c) > 1/(1+c^{-2}\tilde z_1(c)\tilde z_2(c)) > \gamma_1(c):= (1+1.53c^2)/(2.55+1.53c^2) > \kappa,$ $c \in [c_*,c^*]= [0.712\dots, 0.751\dots]. $
 \hfill $\square$
\end{proof}
Next, the graph of $g$ shown in Fig. \ref{F2} was drawn  by using the above mentioned data; it is clear from its shape that $\kappa$ is the global attractor of $g$.  Indeed, the
second iteration $g^2: [g^2(\theta), g(\theta)] \to [g^2(\theta), g(\theta)]$ is a piece-wise linear map, which slopes can not exceed $|k_2k_3|= 0.75$ in the absolute value. Thus all the assumptions of Theorem \ref{main2} are satisfied for all  $c \in [c_*,c^*]$, which proves statements a), b) of Theorem \ref{main1}.
Finally, the part c) follows from \cite[Theorem 3]{TTT}.   \hfill $\square$

\begin{remark} It is comforting  to observe that the conclusions of Theorem \ref{main1} agree with the statement of \cite[Remark 2]{GT} which says that, in the case of existence of non-monotone and non-oscillating wavefronts,
the equation  $z^2-c_*z-1-g'_\kappa e^{-2zc_*} =0$,   where
$$
g'_{\kappa}: = \inf_{x\in
(0,\kappa)}(g(x)-g(\kappa))/(x-\kappa) = -2.69\dots,
$$
can not have negative real roots.
\end{remark}
To illustrate our theoretical results, in Fig. \ref{F1} we are presenting a graph of the minimal wavefront ($q= 0.12\dots$). In its derivation we have used the estimate $q \leq 0.13\dots$ which follows  from (\ref{rep2}), (\ref{sic}).  The graph exhibits only one local extremum.
We believe that it is possible to find $g$ defined by (\ref{dg}) such that the associated wavefront will have two critical points.
It seems that the number of the critical points cannot  exceed 2 (at least for piece-wise linear $g$ defined by     (\ref{dg})).  It is an interesting fact that for $c > c^*$ close to $c^*$ (for example, for $c = 0.8$) the visually observable shape of the wavefront profile $\phi$ is almost the same as it appears in Fig. \ref{F1}. In other words,  the oscillatory nature of 
$\phi$ can be detected only after a suitable magnification of its graph  near the positive equilibrium.

\section*{Acknowledgments}

\vspace{-1.5mm}

\noindent  This research was supported in part by the CONICYT grant 801100006 (A. Ivanov) and the FONDECYT grant 1110309 (S. Trofimchuk and C. Gomez). S. Trofimchuk  also acknowledge support from
CONICYT  PBCT program ACT-56.
C. Gomez was  supported by the CONICYT program
 "Becas para Estudios de Doctorado en Chile". We would like to thank Penn State student Valerie Lindner for her computational and graphical work some of which is used in this paper. She has done this work within a PSU W-B undergraduate student research project.

\end{document}